\documentclass[11pt]{article}
\usepackage{amsmath, amsthm, amssymb, amsfonts}
\usepackage{mathrsfs}
\usepackage{graphicx}
\usepackage{tikz}
\usetikzlibrary{arrows, decorations.markings}
\usepackage{enumerate}
\usepackage{fullpage}

\usepackage[colorlinks=true, linkcolor=blue, citecolor=red, urlcolor=blue]{hyperref}

\theoremstyle{plain}
\newtheorem{theorem}{Theorem}
\newtheorem{lemma}[theorem]{Lemma}
\newtheorem{proposition}[theorem]{Proposition}

\theoremstyle{remark}

\newcommand{\norm}[1]{\left\|#1\right\|}

\newcommand\blfootnote[1]{%
  \begingroup
  \renewcommand\thefootnote{}\footnote{#1}%
  \addtocounter{footnote}{-1}%
  \endgroup
}

\title{{\Large\bf Maximal Averages on the Affine Group $G_n$ and applications}}
\author{Ji Li, Chun-Yen Shen and Chaojie Wen}
\date{}

\begin{document}

\maketitle
\blfootnote{{Keywords:} isotropic affine groups, maximal averages, $L^p$-boundedness, weak type endpoint estimates}
\blfootnote{{Mathematics Subject Classification (2020):} 42B25, 43A80, 22E25.}

\begin{abstract}
Let \(G_n=\mathbb R^n\rtimes\mathbb R_+\) be equipped with the left Haar
measure
\(
        d\mu(x,y)=\frac{dx\,dy}{y^{n+1}}.
\)
We study maximal averages associated with three basic motions on \(G_n\):
horizontal translations, vertical dilations, and fixed hyperbolic geodesics in
the upper half-space model.  The translation maximal operator is the Euclidean
Hardy--Littlewood maximal operator on each horizontal slice.  The
Haar-compatible dilation maximal operator is of weak type \((1,1)\) and bounded on
\(L^p(G_n)\) for \(1<p\le\infty\), but it is not strongly bounded on
\(L^1(G_n)\).  By contrast, the unweighted Lebesgue dilation average is
unbounded on every finite \(L^p(G_n)\) and is not of weak type \((1,1)\).

For fixed hyperbolic geodesic averages, the large-time part is strongly bounded
on \(L^1(G_n)\) because of modular exponential decay.  The small-time part is a
finite-type parabolic maximal problem.  
Using the corresponding local finite-type \(L\log\log L\) endpoint estimate for the geodesic slice, we prove
\(
        \mathcal M_{\gamma_\omega}:L\log\log L(G_n)
        \longrightarrow L^{1,\infty}(G_n)
\)
in weak Orlicz form, together with the strong \(L^p(G_n)\) bounds for
\(1<p\le\infty\).  We also show that the strong \(L^1\) endpoint fails.
Finally, we record a discrete random-walk maximal inequality whose sufficient
condition is expressed through the modular drift
$$
        \rho_p(\sigma)=\int_{G_n}y(h)^{n/p}\,d\sigma(h),
$$
where \(\sigma\) is the probability measure defining the right random walk.
\end{abstract}

\section{Introduction}\label{sec:intro}

Let
\(
        G_n=\mathbb R^n\rtimes\mathbb R_+
\)
be the semidirect product with multiplication
$$
        (x,y)(x',y')=(x+yx',yy'),
        \qquad x,x'\in\mathbb R^n,\quad y,y'>0.
$$
This is the subgroup of the full affine group of $\mathbb R^n$ generated by translations and positive scalar dilations.  We call it the \emph{isotropic affine group}.  It is also an $ax+b$ type group.  We identify $G_n$ with the upper half-space $\mathbb R^n\times\mathbb R_+$.  With the standard left-invariant Riemannian metric
$$
        ds^2=y^{-2}(|dx|^2+dy^2),
$$
this is the upper half-space model of real hyperbolic space $\mathbb H^{n+1}$.

The left Haar measure is
$$
        d\mu(x,y)=\frac{dx\,dy}{y^{n+1}},
$$
and the modular function is
$$
        \Delta(x,y)=y^{-n}.
$$
The factor $y^{-n}$ is the source of the normalizations which appear below.  In particular, right dilation by $(0,e^t)$ changes left Haar measure by the factor $e^{nt}$, and the corresponding Haar-compatible average contains the compensating weight $e^{-nt}$.

The aim of this paper is to give precise $L^p$ and endpoint estimates for maximal averages associated with three simple families of right translations on $G_n$.  The first two families, horizontal translations and vertical dilations, reduce to classical one-parameter Hardy--Littlewood maximal operators.  We give the reductions explicitly, because they determine the correct endpoint statements.  The third family consists of averages along one fixed hyperbolic geodesic.  For this operator the small-time part is a parabolic endpoint problem, while the large-time part has extra decay coming from the modular function.

We consider the following maximal operators.

\begin{enumerate}
    \item \emph{Horizontal translations.}
    For $(x,y)\in G_n$, set
    $$
    \mathcal M_{\mathrm{trans}}f(x,y)
    =
    \sup_{r>0}\frac1{|B_r|}\int_{B_r}|f(x+yt,y)|\,dt,
    $$
    where $B_r\subset\mathbb R^n$ is the Euclidean ball centered at the origin.

    \item \emph{Vertical dilations.}
    The unweighted dilation maximal operator is
    $$
    M_{\mathrm{Leb}}f(x,y)
    =
    \sup_{r>0}\frac1r\int_0^r |f(x,ye^t)|\,dt.
    $$
    This normalization is not compatible with left Haar measure.  The Haar-compatible dilation maximal operator is
    $$
    \mathcal M_{\mathrm{dil}}f(x,y)
    =
    \sup_{r>0}\frac1{V_n(r)}\int_0^r |f(x,ye^t)|e^{-nt}\,dt,
    \qquad
    V_n(r)=\int_0^r e^{-nt}\,dt.
    $$

    \item \emph{Fixed hyperbolic geodesics.}
    For a fixed direction $\omega\in S^{n-1}$, let
    $
        \gamma_\omega(t)=(\omega\tanh t,\operatorname{sech}t),
        \ t\ge0.
    $
    The corresponding maximal operator is
    $$
    \mathcal M_{\gamma_\omega}f(x,y)
    =
    \sup_{r>0}\frac1r\int_0^r
    |f(x+y\omega\tanh t,y\operatorname{sech}t)|\,dt.
    $$
\end{enumerate}

Our first result is the translation estimate.

\begin{theorem}\label{thm:intro-trans}
The translation maximal operator $\mathcal M_{\mathrm{trans}}$ is of weak type $(1,1)$ and is bounded on $L^p(G_n)$ for every $1<p\le\infty$.  It is not bounded from $L^1(G_n)$ to $L^1(G_n)$.
\end{theorem}

The second result gives the endpoint behavior of the two dilation normalizations.

\begin{theorem}\label{thm:intro-dil}
The unweighted dilation maximal operator $M_{\mathrm{Leb}}$ is not bounded on $L^p(G_n)$ for any $1\le p<\infty$, and it is not of weak type $(1,1)$.  The Haar-compatible dilation maximal operator $\mathcal M_{\mathrm{dil}}$ is of weak type $(1,1)$ and is bounded on $L^p(G_n)$ for every $1<p\le\infty$.  However, $\mathcal M_{\mathrm{dil}}$ is not bounded from $L^1(G_n)$ to $L^1(G_n)$.
\end{theorem}

The third result concerns fixed geodesic averages.

\begin{theorem}\label{thm:intro-geo}
For every fixed direction \(\omega\in S^{n-1}\), the geodesic maximal operator
\(\mathcal M_{\gamma_\omega}\) is bounded on \(L^p(G_n)\) for every
\(1<p\le\infty\).  Moreover,
$$
        \mathcal M_{\gamma_\omega}:L\log\log L(G_n)
        \longrightarrow L^{1,\infty}(G_n)
$$
in the weak Orlicz sense: for every \(\lambda>0\),
$$
        \mu\{(x,y):\mathcal M_{\gamma_\omega}f(x,y)>\lambda\}
        \le
        C_n\int_{G_n}
        \frac{|f(x,y)|}{\lambda}
        \log\log\left(e^e+\frac{|f(x,y)|}{\lambda}\right)
        d\mu(x,y).
$$
The operator is not bounded from \(L^1(G_n)\) to \(L^1(G_n)\).
\end{theorem}

The weak \(L\log\log L\) estimate is the endpoint supplied by the finite-type
theorem used in this paper.  We do not claim that it is sharp for this
particular fixed-geodesic operator.  In particular, the failure of strong
\(L^1\) boundedness does not rule out a weak type \((1,1)\) estimate or an
endpoint estimate in a smaller Orlicz space.

We also include a simple random-walk estimate.  Let $\sigma$ be a compactly supported probability measure on $G_n$, and define
$$
        R_\sigma f(g)=\int_{G_n}f(gh)\,d\sigma(h),
        \qquad
        \widetilde M_\sigma f(g)
        =
        \sup_{N\ge1}\frac1N\sum_{k=1}^N |R_\sigma^k f(g)|.
$$
For $1\le p<\infty$, set
$$
        \rho_p(\sigma)=\int_{G_n} y(h)^{n/p}\,d\sigma(h).
$$

\begin{theorem}\label{thm:random_walk_main}
Let $\sigma$ be a compactly supported probability measure on $G_n$.
\begin{enumerate}[\rm(i)]
    \item If $1\le p<\infty$ and $\rho_p(\sigma)<1$, then $\widetilde M_\sigma$ is bounded on $L^p(G_n)$.
    \item If $\rho_1(\sigma)>1$, then $\widetilde M_\sigma$ is not bounded on $L^1(G_n)$ and is not of weak type $(1,1)$.
\end{enumerate}
\end{theorem}

The paper is organized as follows.  Section~\ref{sec:prelim} recalls the group law, Haar measures, and modular function.  Section~\ref{sec:trans} proves the translation estimate by reducing to the Euclidean Hardy--Littlewood maximal operator.  Section~\ref{sec:dil} treats both dilation normalizations.  Section~\ref{sec:geo} proves the fixed-geodesic bounds by reducing the local
part to a finite-type parabolic maximal theorem and by treating the large-time
tail directly; it also includes the strong \(L^1\) counterexample.  Section~\ref{sec:prob} records the Brownian and random-walk interpretations of the same modular drift.

\section{Preliminaries}\label{sec:prelim}

\subsection{The Group Structure}
Let $G_n = \mathbb{R}^n \rtimes \mathbb{R}_+$ be the isotropic affine group, i.e. the subgroup of the full affine group of $\mathbb R^n$ generated by translations and positive scalar dilations (we refer to \cite{LL,M,MT,LVY,V,WY,Z}). We parameterize elements as pairs $(x,y)$ with $x \in \mathbb{R}^n$ (the translation part) and $y > 0$ (the dilation part). The group law is given by the semidirect product law:
\begin{equation}\label{eq:group-law}
(x,y)(x',y') = (x + yx', \, yy').
\end{equation}
The identity element is $(0,1)$. The inverse of an element $(x,y)$ is
\begin{equation}\label{eq:inverse}
(x,y)^{-1} = (-y^{-1}x, \, y^{-1}).
\end{equation}
This can be verified directly:
$$
(x,y)(-y^{-1}x, y^{-1}) = (x + y(-y^{-1}x), \, y(y^{-1})) = (x - x, \, 1) = (0,1).
$$

\subsection{Visualization of Motion}
In the Poincaré half-plane model ($n=1$), these three trajectories originating from a point $(x,y)$ exhibit distinct geometric behaviors.

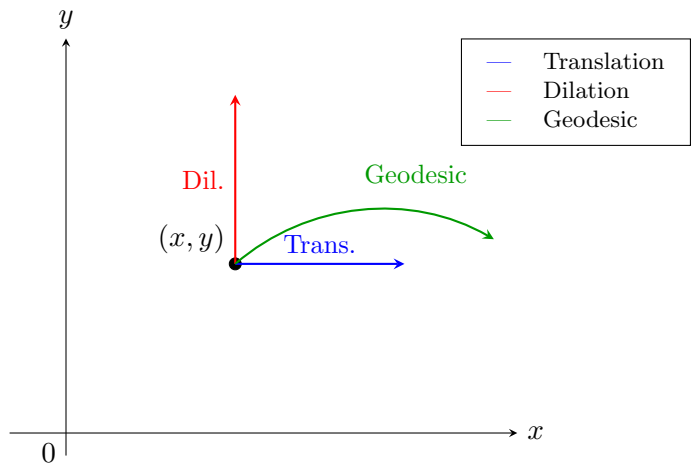
\begin{center}
\begin{tikzpicture}[>=stealth, scale=1.5]
    \draw[->] (-0.5,0) -- (4,0) node[right] {$x$};
    \draw[->] (0,-0.2) -- (0,3.5) node[above] {$y$};
    \node at (0,0) [below left] {$0$};

    \coordinate (P) at (1.5, 1.5);
    \filldraw (P) circle (1.5pt) node[above left] {$(x,y)$};

    \draw[blue, thick, ->] (P) -- (3, 1.5) node[midway, above] {\small Trans.};
    
    \draw[red, thick, ->] (P) -- (1.5, 3) node[midway, left] {\small Dil.};

    \draw[green!60!black, thick, ->] (1.5, 1.5) arc (130.9:60:1.985);
    \node[green!60!black] at (3.1, 2.3) {\small Geodesic};

    \node[anchor=north west, draw, fill=white, font=\footnotesize] at (3.5,3.5) {
        \begin{tabular}{ll}
        \textcolor{blue}{---} & Translation \\
        \textcolor{red}{---} & Dilation \\
        \textcolor{green!60!black}{---} & Geodesic
        \end{tabular}
    };
\end{tikzpicture}
\end{center}

\subsection{Haar Measure and Modular Function}\label{sec:2.3}
The left Haar measure $d\mu$ on $G_n$ is given by
\begin{equation}\label{eq:haar}
d\mu(x,y) = \frac{dx \, dy}{y^{n+1}},
\end{equation}
where $dx$ denotes the standard Lebesgue measure on $\mathbb{R}^n$.

Left invariance is easily checked. For $g=(x_0, y_0)$, the left translation $L_g(x,y) = (x_0+y_0x, y_0y)$ has Jacobian determinant $y_0^{n+1}$. The measure transforms as:
$$
\frac{d(x_0+y_0x) \, d(y_0y)}{(y_0y)^{n+1}} = \frac{y_0^n dx \, y_0 dy}{y_0^{n+1} y^{n+1}} = \frac{dx \, dy}{y^{n+1}}.
$$

The group $G_n$ is non-unimodular. The right Haar measure is
$$
        d\mu_R(x,y)=\frac{dx\,dy}{y}.
$$
We use the convention
$
        d\mu_R(g)=\Delta_{G_n}(g)^{-1}\,d\mu(g).
$
Thus
\begin{equation}\label{eq:modular}
        \Delta_{G_n}(x,y)=y^{-n}.
\end{equation}
If $h=(a,b)\in G_n$, then right translation satisfies
\begin{equation}\label{eq:right-translation-measure}
        d\mu(g h)=b^{-n}\,d\mu(g).
\end{equation}
Equivalently, for every non-negative measurable function $\Phi$,
\begin{equation}\label{eq:right-translation-integral}
        \int_{G_n}\Phi(g h)\,d\mu(g)
        =b^n\int_{G_n}\Phi(g)\,d\mu(g).
\end{equation}
This last identity is the form used in the estimates for right convolution operators.

\section{Translation Averages}\label{sec:trans}

The translation subgroup is
$$
        H_{\mathrm{trans}}=\{(t,1):t\in\mathbb R^n\}.
$$
Right multiplication by $(t,1)$ gives
$
        (x,y)(t,1)=(x+yt,y).
$
Thus the vertical coordinate is fixed and the horizontal coordinate is translated by the vector $yt$.

Let $B_r=\{u\in\mathbb R^n:|u|<r\}$ and let $|B_r|$ denote its Euclidean volume. For $f\in L^1_{\mathrm{loc}}(G_n)$, define the translation maximal operator
$$
\mathcal M_{\mathrm{trans}}f(x,y)
=
\sup_{r>0}\frac1{|B_r|}\int_{B_r}|f(x+yt,y)|\,dt.
$$

\begin{theorem}\label{thm:trans}
For every $n\ge1$, the operator $\mathcal M_{\mathrm{trans}}$ is of weak type $(1,1)$ and is bounded on $L^p(G_n)$ for every $1<p\le\infty$. More precisely, for $1<p\le\infty$,
$$
\norm{\mathcal M_{\mathrm{trans}}f}_{L^p(G_n)}
\le
C_{p,n}\norm{f}_{L^p(G_n)},
$$
where $C_{p,n}$ may be taken to be the $L^p(\mathbb R^n)$ norm of the Euclidean Hardy--Littlewood maximal operator.
\end{theorem}

\begin{proof}
Fix $y>0$ and write
$
        f_y(x)=f(x,y),\ x\in\mathbb R^n.
$
For $r>0$ we have
$$
A_rf(x,y)
=
\frac1{|B_r|}\int_{B_r}|f_y(x+yt)|\,dt.
$$
Set $u=x+yt$. Then $du=y^n\,dt$, so $dt=y^{-n}\,du$. As $t$ ranges over $B_r$, the variable $u$ ranges over $B_{yr}(x)$. Since
$
        |B_r|=y^{-n}|B_{yr}|,
$
we get
\begin{align*}
A_rf(x,y)
&=
\frac1{y^{-n}|B_{yr}|}
\int_{B_{yr}(x)}|f_y(u)|y^{-n}\,du =
\frac1{|B_{yr}|}\int_{B_{yr}(x)}|f_y(u)|\,du.
\end{align*}
Taking the supremum over $r>0$ is the same as taking the supremum over $R=yr>0$. Hence
\begin{equation}\label{eq:trans-slice-reduction}
        \mathcal M_{\mathrm{trans}}f(x,y)
        =M_{\mathrm{HL}}^{(n)}(f_y)(x),
\end{equation}
where $M_{\mathrm{HL}}^{(n)}$ is the centered Euclidean Hardy--Littlewood maximal operator on $\mathbb R^n$.

For $1<p<\infty$, Fubini's theorem and the classical Hardy--Littlewood theorem give
\begin{align*}
\norm{\mathcal M_{\mathrm{trans}}f}_{L^p(G_n)}^p
&=
\int_0^\infty
\int_{\mathbb R^n}|M_{\mathrm{HL}}^{(n)}f_y(x)|^p\,dx\,\frac{dy}{y^{n+1}} \\
&\le
C_{p,n}^p
\int_0^\infty
\int_{\mathbb R^n}|f(x,y)|^p\,dx\,\frac{dy}{y^{n+1}} \\
&=
C_{p,n}^p
\norm{f}_{L^p(G_n)}^p.
\end{align*}
The case $p=\infty$ is immediate from the definition.

The weak-type endpoint is just as direct. For $\lambda>0$, using \eqref{eq:trans-slice-reduction} and the Euclidean weak $(1,1)$ estimate,
\begin{align*}
\mu\{(x,y):\mathcal M_{\mathrm{trans}}f(x,y)>\lambda\}
&=
\int_0^\infty
\bigl|\{x\in\mathbb R^n:M_{\mathrm{HL}}^{(n)}f_y(x)>\lambda\}\bigr|
\frac{dy}{y^{n+1}} \\
&\le
\frac{C_n}{\lambda}
\int_0^\infty
\int_{\mathbb R^n}|f(x,y)|\,dx\,\frac{dy}{y^{n+1}} \\
&=
\frac{C_n}{\lambda}\norm{f}_{L^1(G_n)}.
\end{align*}
This proves the theorem.
\end{proof}

\begin{proposition}
\label{prop:trans-not-L1}
The translation maximal operator \(\mathcal M_{\mathrm{trans}}\) is not bounded
from \(L^1(G_n)\) to \(L^1(G_n)\).
\end{proposition}

\begin{proof}
We use the same slice reduction as above and a standard test function for the
Euclidean Hardy--Littlewood maximal operator.  Let \(0<\varepsilon<1/4\), and set
$$
        f_\varepsilon(x,y)
        =
        \mathbf 1_{\{|x|<\varepsilon\}}(x)
        \mathbf 1_{[1,2]}(y).
$$
Then
$$
        \|f_\varepsilon\|_{L^1(G_n)}
        =
        |B(0,\varepsilon)|
        \int_1^2 \frac{dy}{y^{n+1}}
        \simeq_n \varepsilon^n .
$$
For each fixed \(y\in[1,2]\), the horizontal slice is
\(
        (f_\varepsilon)_y=\mathbf 1_{B(0,\varepsilon)}.
\)
If \(2\varepsilon<|x|<1\), choose the Euclidean ball centered at \(x\) with
radius \(|x|+\varepsilon\).  This ball contains \(B(0,\varepsilon)\).  Hence
$$
        M_{\mathrm{HL}}^{(n)}\mathbf 1_{B(0,\varepsilon)}(x)
        \ge
        \frac{|B(0,\varepsilon)|}{|B(x,|x|+\varepsilon)|}
        \ge
        c_n\frac{\varepsilon^n}{|x|^n}.
$$
Using the exact identity
\(
        \mathcal M_{\mathrm{trans}}f(x,y)
        =M_{\mathrm{HL}}^{(n)}(f_y)(x),
\)
we obtain
$$
\begin{aligned}
        \|\mathcal M_{\mathrm{trans}}f_\varepsilon\|_{L^1(G_n)}
        &\ge
        c_n
        \int_1^2 \frac{dy}{y^{n+1}}
        \int_{2\varepsilon<|x|<1}
        \frac{\varepsilon^n}{|x|^n}\,dx                                      
        \ge
        c_n\varepsilon^n\log\frac1\varepsilon .
\end{aligned}
$$
If a strong \(L^1\) estimate held, then the last lower bound would be bounded
by \(C\|f_\varepsilon\|_{L^1(G_n)}\le C_n\varepsilon^n\) uniformly in
\(\varepsilon\).  This is impossible as \(\varepsilon\downarrow0\).  Therefore
\(\mathcal M_{\mathrm{trans}}\) is not strongly bounded on \(L^1(G_n)\).
\end{proof}

\section{Dilation Averages}\label{sec:dil}

The dilation subgroup is
$$
        H_{\mathrm{dil}}=\{(0,e^t):t\in\mathbb R\}.
$$
Right multiplication by $(0,e^t)$ gives
$
        (x,y)(0,e^t)=(x,ye^t).
$
The parameter $t$ is therefore the logarithmic vertical displacement.

\subsection{The failure of standard averaging via Lebesgue measure}

We first record that the unweighted average in the $t$-parameter is not compatible with the left Haar measure. Define
$$
M_{\mathrm{Leb}}f(x,y)
=
\sup_{r>0}\frac1r\int_0^r |f(x,ye^t)|\,dt.
$$

\begin{lemma}\label{lem:growth}
For every $n\ge1$ and every $1\le p<\infty$,
$$
        \int_0^\infty \frac{e^{nu}}{(1+u)^p}\,du=\infty.
$$
\end{lemma}

\begin{proof}
Let
$$
        \phi(u)=\frac{e^{nu}}{(1+u)^p}.
$$
Then
$$
        \log\phi(u)=nu-p\log(1+u)
        =u\left(n-p\frac{\log(1+u)}u\right).
$$
Since $\log(1+u)/u\to0$ as $u\to\infty$, the quantity in parentheses tends to $n>0$. Hence $\log\phi(u)\to\infty$, and so $\phi(u)\to\infty$. In particular, $\phi(u)\ge1$ for all sufficiently large $u$, and the integral diverges.
\end{proof}

\begin{proposition}\label{prop:leb_fail}
For every $n\ge1$ and every $1\le p<\infty$, the operator $M_{\mathrm{Leb}}$ is not bounded on $L^p(G_n)$.
\end{proposition}

\begin{proof}
Let
$
        K=[0,1]^n\times[1,e]
$
and set $f=\mathbf 1_K$. Then
$$
\norm{f}_{L^p(G_n)}^p
=
\int_{[0,1]^n}\int_1^e \frac{dy}{y^{n+1}}\,dx
=
\frac{1-e^{-n}}n<\infty.
$$

Now take $x\in[0,1]^n$ and $0<y<1$. The condition $f(x,ye^t)=1$ is equivalent to
$$
        1\le ye^t\le e,
$$
or
$$
        \log(1/y)\le t\le 1+\log(1/y).
$$
Put $T_y=\log(1/y)>0$ and choose $r=T_y+1$. Then
$$
\frac1r\int_0^r f(x,ye^t)\,dt
=
\frac1{T_y+1}\int_{T_y}^{T_y+1}1\,dt
=
\frac1{1+\log(1/y)}.
$$
Thus
$$
        M_{\mathrm{Leb}}f(x,y)
        \ge
        \frac1{1+\log(1/y)}
        =
        \frac1{1-\log y}
$$
for $x\in[0,1]^n$ and $0<y<1$.
Therefore
\begin{align*}
\norm{M_{\mathrm{Leb}}f}_{L^p(G_n)}^p
&\ge
\int_{[0,1]^n}\int_0^1
\frac1{(1-\log y)^p}\frac{dy}{y^{n+1}}\,dx.
\end{align*}
With the change of variables $u=-\log y$, so that $dy/y=-du$, this lower bound becomes
$$
        \int_0^\infty \frac{e^{nu}}{(1+u)^p}\,du,
$$
which is infinite by Lemma~\ref{lem:growth}. Hence $M_{\mathrm{Leb}}f\notin L^p(G_n)$.
\end{proof}

\begin{proposition}
\label{prop:leb-weak-fail}
For every \(n\ge1\), the operator \(M_{\mathrm{Leb}}\) is not of weak type
\((1,1)\) on \(G_n\).
\end{proposition}

\begin{proof}
Use the same function as in the previous proof,
$
        f=\mathbf 1_{[0,1]^n\times[1,e]} .
$
Then \(f\in L^1(G_n)\).  The previous proof showed that, for
\(x\in[0,1]^n\) and \(0<y<1\),
$$
        M_{\mathrm{Leb}}f(x,y)
        \ge
        \frac1{1+\log(1/y)} .
$$
Let \(0<\lambda<1/2\).  If
$$
        0<\log(1/y)<\frac1\lambda-1,
$$
then
\(
        (1+\log(1/y))^{-1}>\lambda.
\)
Equivalently,
\(
        e^{1-1/\lambda}<y<1.
\)
Therefore
$$
\begin{aligned}
        \mu\{(x,y):M_{\mathrm{Leb}}f(x,y)>\lambda\}
        &\ge
        \int_{[0,1]^n}\int_{e^{1-1/\lambda}}^1
        \frac{dy}{y^{n+1}}\,dx                                      
        =
        \frac{e^{n(1/\lambda-1)}-1}{n} .
\end{aligned}
$$
Consequently
$$
        \lambda\,
        \mu\{(x,y):M_{\mathrm{Leb}}f(x,y)>\lambda\}
        \ge
        \frac{\lambda}{n}
        \left(e^{n(1/\lambda-1)}-1\right),
$$
and the right hand side tends to infinity as \(\lambda\downarrow0\).  Since
\(\|f\|_{L^1(G_n)}<\infty\) is fixed, no weak type \((1,1)\) estimate can hold.
\end{proof}

\subsection{The Haar-compatible dilation maximal operator}

The correct normalization for right dilation averages is obtained from the modular factor. Since
$$
        \Delta(0,e^t)=e^{-nt},
$$
we average against $e^{-nt}\,dt$. Define
$$
        V_n(r)=\int_0^r e^{-nt}\,dt=\frac{1-e^{-nr}}n.
$$

For $f\in L^1_{\mathrm{loc}}(G_n)$, define Haar-compatible dilation maximal operator
$$
\mathcal M_{\mathrm{dil}}f(x,y)
=
\sup_{r>0}\frac1{V_n(r)}\int_0^r |f(x,ye^t)|e^{-nt}\,dt.
$$
The point of this normalization is made transparent by logarithmic coordinates. Put
$$
        u=\log y,
        \qquad
        F(x,u)=f(x,e^u).
$$
Then
\begin{equation}\label{eq:haar-log}
        d\mu(x,y)=\frac{dx\,dy}{y^{n+1}}
        =e^{-nu}\,dx\,du.
\end{equation}
Let
$$
        d\nu(u)=e^{-nu}\,du.
$$
For a function $g$ on $\mathbb R$, define the one-sided maximal operator
$$
        M_\nu^+g(u)
        =
        \sup_{r>0}
        \frac1{\nu([u,u+r])}
        \int_u^{u+r}|g(v)|\,d\nu(v).
$$
Since
$$
\int_u^{u+r}|F(x,v)|e^{-nv}\,dv
=
e^{-nu}\int_0^r |F(x,u+t)|e^{-nt}\,dt
$$
and
$$
        \nu([u,u+r])=e^{-nu}V_n(r),
$$
we have the exact identity
\begin{equation}\label{eq:dilation-reduction}
        \mathcal M_{\mathrm{dil}}f(x,e^u)
        =
        M_\nu^+(F(x,\cdot))(u).
\end{equation}

\begin{lemma}\label{lem:weighted-one-sided-HL}
The operator $M_\nu^+$ is of weak type $(1,1)$ on $(\mathbb R,d\nu)$ and is bounded on $L^p(\mathbb R,d\nu)$ for every $1<p\le\infty$.
\end{lemma}

\begin{proof}
Make the change of variables
$$
        s=\frac{e^{-nu}}n.
$$
Then $ds=-e^{-nu}\,du=-d\nu(u)$. Thus, up to reversal of orientation, $d\nu(u)$ becomes Lebesgue measure $ds$ on $(0,\infty)$. If
$$
        G(s)=g\left(-\frac1n\log(ns)\right),
$$
then
$$
        \int_{\mathbb R}|g(u)|^p\,d\nu(u)
        =
        \int_0^\infty |G(s)|^p\,ds
$$
for $1\le p<\infty$, and the corresponding $L^\infty$ norms are equal.

Moreover, the interval $[u,u+r]$ is transformed into the interval
$
        [se^{-nr},s]
        \subset(0,\infty).
$
Therefore
$$
        M_\nu^+g(u)
        =
        \sup_{0<a<s}\frac1{s-a}\int_a^s |G(\sigma)|\,d\sigma.
$$
The right-hand side is the usual one-sided Hardy--Littlewood maximal function on the half-line, applied to $G$ and evaluated at $s$. Extending $G$ by zero to the whole real line, it is dominated by the usual uncentered Hardy--Littlewood maximal operator on $\mathbb R$. The classical weak $(1,1)$ and strong $L^p$, $1<p\le\infty$, estimates therefore give the desired bounds after changing variables back.
\end{proof}

\begin{theorem}\label{thm:dil}
The Haar-compatible dilation maximal operator $\mathcal M_{\mathrm{dil}}$ is of weak type $(1,1)$ and is bounded on $L^p(G_n)$ for every $1<p\le\infty$.
\end{theorem}

\begin{proof}
For $1<p<\infty$, combine \eqref{eq:dilation-reduction}, Lemma~\ref{lem:weighted-one-sided-HL}, and Fubini's theorem:
\begin{align*}
\norm{\mathcal M_{\mathrm{dil}}f}_{L^p(G_n)}^p
&=
\int_{\mathbb R^n}\int_{\mathbb R}
|M_\nu^+(F(x,\cdot))(u)|^p\,d\nu(u)\,dx \\
&\le
C_p^p
\int_{\mathbb R^n}\int_{\mathbb R}|F(x,u)|^p\,d\nu(u)\,dx \\
&=
C_p^p\norm{f}_{L^p(G_n)}^p.
\end{align*}
The case $p=\infty$ is immediate from the definition.

For the weak-type endpoint, again by \eqref{eq:dilation-reduction} and Lemma~\ref{lem:weighted-one-sided-HL},
\begin{align*}
\mu\{(x,y):\mathcal M_{\mathrm{dil}}f(x,y)>\lambda\}
&=
\int_{\mathbb R^n}
\nu\{u:M_\nu^+(F(x,\cdot))(u)>\lambda\}\,dx \\
&\le
\frac{C}{\lambda}
\int_{\mathbb R^n}\int_{\mathbb R}|F(x,u)|\,d\nu(u)\,dx \\
&=
\frac{C}{\lambda}\norm{f}_{L^1(G_n)}.
\end{align*}
This proves weak type $(1,1)$.
\end{proof}

\subsection{\texorpdfstring{Failure of strong type $L^1$}{Failure of strong type L1}}

The weak endpoint just proved does not imply strong $L^1$ boundedness. In fact strong type $(1,1)$ fails.

\begin{proposition}\label{prop:dil_endpoint_fail}
The operator $\mathcal M_{\mathrm{dil}}$ is not bounded from $L^1(G_n)$ to $L^1(G_n)$.
\end{proposition}

\begin{proof}
Let
$$
        f(x,y)=y^n\mathbf 1_{[0,1]^n}(x)\mathbf 1_{[1,e]}(y).
$$
Equivalently, in logarithmic coordinates,
$$
        F(x,u)=e^{nu}\mathbf 1_{[0,1]^n}(x)\mathbf 1_{[0,1]}(u).
$$
Then
$$
\norm{f}_{L^1(G_n)}
=
\int_{[0,1]^n}\int_1^e y^n\frac{dy}{y^{n+1}}\,dx
=
\int_1^e\frac{dy}{y}=1.
$$

Now fix $x\in[0,1]^n$ and $0<y<1$, and write $u=\log y<0$. Choose $r=1-u$. Then the interval $[u,u+r]$ contains $[0,1]$, and hence
\begin{align*}
\int_0^r |f(x,ye^t)|e^{-nt}\,dt
&\ge
\int_{-u}^{1-u} e^{n(u+t)}e^{-nt}\,dt 
=
\int_{-u}^{1-u} e^{nu}\,dt
=
e^{nu}.
\end{align*}
Since $V_n(r)\le 1/n$, we obtain
$$
        \mathcal M_{\mathrm{dil}}f(x,y)
        \ge
        n e^{nu}=n y^n
$$
for all $x\in[0,1]^n$ and $0<y<1$. Therefore
\begin{align*}
\norm{\mathcal M_{\mathrm{dil}}f}_{L^1(G_n)}
&\ge
\int_{[0,1]^n}\int_0^1 n y^n\frac{dy}{y^{n+1}}\,dx 
=
 n\int_0^1\frac{dy}{y}
 =\infty.
\end{align*}
Thus $\mathcal M_{\mathrm{dil}}$ is not strong type $L^1$.
\end{proof}

\section{Fixed Geodesic Averages}\label{sec:geo}

Fix a unit vector $\omega\in S^{n-1}$.  The hyperbolic geodesic starting at $(0,1)$ with initial horizontal direction $\omega$ is
$$
        \gamma_\omega(t)=(\omega\tanh t,\operatorname{sech}t),
        \qquad t\ge0.
$$
The parameter \(t\) is hyperbolic arc length.  Indeed,
$$
        |\gamma_\omega'(t)|_{\mathrm{Eucl}}
        =\operatorname{sech}t,
$$
while the height of \(\gamma_\omega(t)\) is
\(y(t)=\operatorname{sech}t\).  Hence the speed of \(\gamma_\omega\) for the
metric \(ds^2=y^{-2}(|dx|^2+dy^2)\) is equal to \(1\).
Right multiplication gives
$$
        (x,y)\gamma_\omega(t)
        =
        (x+y\omega\tanh t,y\operatorname{sech}t).
$$

For $f\in L^1_{\mathrm{loc}}(G_n)$, define the fixed geodesic maximal operator
$$
        \mathcal M_{\gamma_\omega}f(x,y)
        =
        \sup_{r>0}\frac1r\int_0^r
        |f(x+y\omega\tanh t,y\operatorname{sech}t)|\,dt.
$$

We shall use the time-$t$ operator
$$
        S_tf(x,y)=f(x+y\omega\tanh t,y\operatorname{sech}t).
$$
The basic calculation is the following.

\begin{lemma}\label{lem:geodesic-norm}
For $1\le p<\infty$,
$$
        \|S_tf\|_{L^p(G_n)}=(\operatorname{sech}t)^{n/p}\|f\|_{L^p(G_n)}.
$$
For $p=\infty$, one has
$$
        \|S_tf\|_{L^\infty(G_n)}\le\|f\|_{L^\infty(G_n)}.
$$
\end{lemma}

\begin{proof}
For $1\le p<\infty$,
$$
\|S_tf\|_{L^p(G_n)}^p
=
\int_0^\infty\int_{\mathbb R^n}
|f(x+y\omega\tanh t,y\operatorname{sech}t)|^p
\frac{dx\,dy}{y^{n+1}}.
$$
Set
$$
        x'=x+y\omega\tanh t,
        \qquad
        y'=y\operatorname{sech}t.
$$
Then
$$
        y=y'\cosh t,
        \qquad
        x=x'-y'\omega\sinh t.
$$
The Jacobian of the inverse change of variables is $\cosh t$.  Hence
$$
        dx\,dy=\cosh t\,dx'\,dy',
        \qquad
        y^{n+1}=(y')^{n+1}(\cosh t)^{n+1}.
$$
Substitution gives
$$
\|S_tf\|_{L^p(G_n)}^p
=
(\operatorname{sech}t)^n
\|f\|_{L^p(G_n)}^p.
$$
This proves the formula.  The $L^\infty$ estimate is immediate.
\end{proof}

We split the maximal operator into a local part and a large-time part:
$$
        \mathcal M_{\gamma_\omega}^{\mathrm{loc}}f
        =
        \sup_{0<r\le1}\frac1r\int_0^r |S_tf|\,dt,
        \qquad
        \mathcal M_{\gamma_\omega}^{\mathrm{tail}}f
        =
        \sup_{r\ge1}\frac1r\int_0^r |S_tf|\,dt.
$$

\subsection{The local part}

We now first prove the local endpoint for the fixed-geodesic operator by reducing it
to the local finite-type parabolic maximal theorem of
Seeger--Tao--Wright~\cite{STW}.  The point is that, in suitable local
coordinates, the fixed-geodesic curve has the same finite-type structure as the
parabola.

\begin{theorem}
\label{thm:geo-local-Lloglog}
For every fixed \(\omega\in S^{n-1}\), every measurable \(f\), and every
\(\lambda>0\),
$$
        \mu\left\{(x,y):
        \mathcal M_{\gamma_\omega}^{\mathrm{loc}}f(x,y)>\lambda
        \right\}
        \le
        C_n
        \int_{G_n}
        \Phi\left(\frac{|f(x,y)|}{\lambda}\right)
        d\mu(x,y).
$$
\end{theorem}

To prove this theorem, we first introduce the slice notation.
Because \(\omega\in S^{n-1}\) is a unit vector, every
\(x\in\mathbb R^n\) has the unique orthogonal decomposition
$$
        x=(x\cdot\omega)\omega+
        \bigl(x-(x\cdot\omega)\omega\bigr).
$$
Thus we write
$$
        u=x\cdot\omega,
        \qquad
        z=x-u\omega\in\omega^\perp,
$$
and hence
$$
        x=u\omega+z.
$$

The map
$$
        (u,z)\in\mathbb R\times\omega^\perp
        \longmapsto
        u\omega+z\in\mathbb R^n
$$
is an orthogonal linear change of variables. Therefore its Jacobian is equal
to \(1\), and Lebesgue measure decomposes as
$$
        dx=du\,dz.
$$
If we define
$$
        d\nu_n(u,y)=\frac{du\,dy}{y^{n+1}},
$$
then
$$
        d\mu(x,y)=dz\,d\nu_n(u,y).
$$

For the fixed-geodesic slice, define
$$
        \mathfrak M_{\mathrm{geo}}F(u,y)
        =
        \sup_{0<r\le1}
        \frac1r\int_0^r
        |F(u+y\tanh t,y\operatorname{sech}t)|\,dt .
$$
The local fixed-geodesic operator leaves the transverse variable \(z\) fixed.
Indeed, if \(x=u\omega+z\), then
$$
        x+y\omega\tanh t
        =
        (u+y\tanh t)\omega+z.
$$
Therefore
$$
        \mathcal M_{\gamma_\omega}^{\mathrm{loc}}f(u\omega+z,y)
        =
        \mathfrak M_{\mathrm{geo}}(f_z)(u,y).
$$

We shall use the following local endpoint estimate.  It is the form of the
Seeger--Tao--Wright near-\(L^1\) theorem needed in the present argument.  The key tool we use is Theorem~1.1 of Seeger--Tao--Wright~\cite{STW}, together
with the local Orlicz formulation discussed in their Section~1.3.1 and the
parabolic example in their Section~1.3.5.  Since our curve is a smooth
variable-coefficient perturbation of the parabolic model on each compact
\(Y\)-strip, we include the short verification of the finite-type geometry and
of the uniformity of the constants.
The local finite-type theorem is applied to a compact family of curves, uniformly controlled on each $Y$-strip by the uniform lower bound on the determinant and by finitely many $C^N$-norms of the normalized curve family.

\begin{theorem}
\label{thm:STW-local-finite-type}
Let \(0<a<b<\infty\).  Define
$$
        \mathcal A H(U,Y)
        =
        \sup_{0<r\le1}
        \frac1r\int_0^r
        |H(U+Y\tanh t,Y\operatorname{sech}t)|\,dt .
$$
Then, for every measurable \(H\) and every \(\lambda>0\),
$$
        \left|
        \left\{(U,Y):a\le Y\le b,
        \mathcal A H(U,Y)>\lambda
        \right\}
        \right|                                        \le
        C_{a,b}
        \int_{\mathbb R^2}
        \Phi\left(\frac{|H(U,Y)|}{\lambda}\right)\,dU\,dY,
$$
where throughout this paper
$
        \Phi(s)=s\log\log(e^e+s),\ s\ge0.
$
The constant depends only on \(a,b\) and on finitely many \(C^N\)-bounds for the
normalized curve family on the strip \(a\le Y\le b\).
\end{theorem}

\begin{proof}
We first verify the finite-type geometry.  Make the smooth change of parameter
$$
        s=\tanh t,
        \qquad 0\le s\le \tanh 1.
$$
If \(\rho=\tanh r\), then \(r\simeq \rho\) for \(0<r\le1\), and
$$
        dt=\frac{ds}{1-s^2}.
$$
The weight \((1-s^2)^{-1}\) is smooth and bounded above and below on
\([0,\tanh1]\).  Hence \(\mathcal A H\) is dominated, up to an absolute
constant, by the corresponding maximal operator with parameter \(s\):
$$
        \widetilde{\mathcal A}H(U,Y)
        =
        \sup_{0<\rho\le\tanh1}
        \frac1\rho\int_0^\rho
        \left|H\left(U+Ys,Y\sqrt{1-s^2}\right)\right|\,ds .
$$
The associated curve is
$$
        \widetilde\Gamma((U,Y),s)
        =
        \left(U+Ys,\,Y\sqrt{1-s^2}\right),
        \qquad 0\le s\le \tanh1.
$$
Its displacement from the base point is
$$
        \widetilde\Theta((U,Y),s)
        =
        \left(Ys,\,Y(\sqrt{1-s^2}-1)\right).
$$
Therefore
$$
        \partial_s\widetilde\Theta((U,Y),s)
        =
        \left(Y,\,-\frac{Ys}{\sqrt{1-s^2}}\right)
$$
and
$$
        \partial_s^2\widetilde\Theta((U,Y),s)
        =
        \left(0,\,-\frac{Y}{(1-s^2)^{3/2}}\right).
$$
Thus
$$
        \det\left(
        \partial_s\widetilde\Theta((U,Y),s),
        \partial_s^2\widetilde\Theta((U,Y),s)
        \right)
        =
        -\frac{Y^2}{(1-s^2)^{3/2}}.
$$
On every strip \(a\le Y\le b\), \(0\le s\le\tanh1\), this determinant is
bounded away from zero and all \(C^N\)-norms of the curve family are bounded by
constants depending only on \(a,b,N\).  Hence the family is a compact smooth
finite-type perturbation of the parabolic model \(s\mapsto (s,s^2)\), uniformly
on the strip.

We now apply the local finite-type form of the Seeger--Tao--Wright \cite{STW}
near-\(L^1\) estimate.  The estimate is local in the base variables.  We first
state the local patch estimate that follows from the preceding finite-type
calculation.

Fix an interval \(I\subset\mathbb R\) centered at the origin and with length
\(
        L=8b,
\)
and define its enlargement
$$
        I^*
        =
        \{U\in\mathbb R:\operatorname{dist}(U,I)<b\}.
$$
Then the local Seeger--Tao--Wright \cite{STW} estimate gives
\begin{equation}
\label{eq:STW-local-U-patch}
\begin{aligned}
        &
        \left|
        \left\{
        (U,Y)\in I\times[a,b]:
        \widetilde{\mathcal A}H(U,Y)>\lambda
        \right\}
        \right|                                           
                \le
        C_{a,b}
        \int_{I^*\times[a\operatorname{sech}1,b]}
        \Phi\left(\frac{|H(U,Y)|}{\lambda}\right)\,dU\,dY .
\end{aligned}
\end{equation}
Here
$$
        \widetilde{\mathcal A}H(U,Y)
        =
        \sup_{0<\rho\le\tanh1}
        \frac1\rho
        \int_0^\rho
        \left|
        H\left(U+Ys,Y\sqrt{1-s^2}\right)
        \right|\,ds .
$$
Indeed, if \((U,Y)\in I\times[a,b]\) and \(0\le s\le\tanh1\), then
$$
        U+Ys\in I^*
$$
because \(0\le Ys< b\).  Also
$$
        a\operatorname{sech}1
        \le
        Y\sqrt{1-s^2}
        \le
        b .
$$
Therefore, on \(I\times[a,b]\),
$$
        \widetilde{\mathcal A}H
        =
        \widetilde{\mathcal A}
        \left(
        H\mathbf 1_{I^*\times[a\operatorname{sech}1,b]}
        \right).
$$
The finite-type determinant estimate proved above is uniform on
\(I\times[a,b]\), and the constants are independent of the position of \(I\)
because the curve family is translation-invariant in the \(U\)-variable.
This proves \eqref{eq:STW-local-U-patch}.

Now we patch the \(U\)-axis.  Let
$$
        I_m=[mL,(m+1)L),
        \qquad m\in\mathbb Z,
$$
and let
$$
        I_m^*
        =
        \{U\in\mathbb R:\operatorname{dist}(U,I_m)<b\}.
$$
The intervals \(I_m\) are disjoint and cover \(\mathbb R\).  Since
\(L=8b\), each enlarged interval is
$$
        I_m^*=(mL-b,(m+1)L+b)
$$
and has length \(L+2b=10b<2L\).  Therefore only two consecutive enlarged
intervals can contain the same point, and hence
$$
        \sum_{m\in\mathbb Z}\mathbf 1_{I_m^*}(U)
        \le
        2
        \qquad\text{for every }U\in\mathbb R .
$$

Set
$$
        E
        =
        \left\{
        (U,Y):a\le Y\le b,\ 
        \widetilde{\mathcal A}H(U,Y)>\lambda
        \right\}.
$$
Then
$$
        E
        =
        \bigcup_{m\in\mathbb Z}
        \left(E\cap(I_m\times[a,b])\right),
$$
and the union is disjoint up to endpoints.  Hence, using
\eqref{eq:STW-local-U-patch},
$$
\begin{aligned}
        |E|
        &\le
        \sum_{m\in\mathbb Z}
        \left|
        E\cap(I_m\times[a,b])
        \right|                                           \\
        &\le
        C_{a,b}
        \sum_{m\in\mathbb Z}
        \int_{I_m^*\times[a\operatorname{sech}1,b]}
        \Phi\left(\frac{|H(U,Y)|}{\lambda}\right)\,dU\,dY  \\
        &=
        C_{a,b}
        \int_{\mathbb R^2}
        \left(\sum_{m\in\mathbb Z}\mathbf 1_{I_m^*}(U)\right)
        \mathbf 1_{[a\operatorname{sech}1,b]}(Y)
        \Phi\left(\frac{|H(U,Y)|}{\lambda}\right)\,dU\,dY  \\
        &\le
        C_{a,b}
        \int_{\mathbb R^2}
        \Phi\left(\frac{|H(U,Y)|}{\lambda}\right)\,dU\,dY .
\end{aligned}
$$
This proves the desired global estimate for
\(\widetilde{\mathcal A}\).  Since the original operator
\(\mathcal A\) is dominated by \(\widetilde{\mathcal A}\), up to an absolute
constant coming from the smooth change of parameter \(s=\tanh t\), the same
estimate holds for \(\mathcal A\), after changing the value of \(C_{a,b}\).
This proves the theorem.
\end{proof}

\begin{theorem}
\label{thm:STW-local-finite-type-Lp}
Let \(0<a<b<\infty\), and define
$$
        \mathcal A H(U,Y)
        =
        \sup_{0<r\le1}
        \frac1r\int_0^r
        |H(U+Y\tanh t,Y\operatorname{sech}t)|\,dt .
$$
Then, for every \(1<p\le\infty\),
$$
        \|\mathcal A H\|_{L^p(\mathbb R\times[a,b])}
        \le
        C_{p,a,b}\|H\|_{L^p(\mathbb R^2)} .
$$
The constant depends only on \(p,a,b\) and on the same finite-type constants
which occur in Theorem~\ref{thm:STW-local-finite-type}.
\end{theorem}

\begin{proof}
The case \(p=\infty\) is immediate from the definition, since
$$
        \mathcal A H(U,Y)\le \|H\|_{L^\infty(\mathbb R^2)} .
$$
Let \(1<p<\infty\).  We derive the estimate from the endpoint estimate in
Theorem~\ref{thm:STW-local-finite-type} and the trivial \(L^\infty\) bound.
This keeps the estimate exactly in the form used above.

For \(\lambda>0\), write
$$
        H_{\lambda}=H\mathbf 1_{\{|H|>\lambda/2\}},
        \qquad
        H^{\lambda}=H\mathbf 1_{\{|H|\le\lambda/2\}}.
$$
Then \(|H^{\lambda}|\le\lambda/2\), and hence
$$
        \mathcal A H^{\lambda}(U,Y)\le \lambda/2 .
$$
By sublinearity,
$$
        \{(U,Y):a\le Y\le b,\ \mathcal A H(U,Y)>\lambda\}
        \subset
        \{(U,Y):a\le Y\le b,\ \mathcal A H_{\lambda}(U,Y)>\lambda/2\}.
$$
Applying Theorem~\ref{thm:STW-local-finite-type} to \(H_{\lambda}\), we get
$$
        \left|
        \{(U,Y):a\le Y\le b,
        \mathcal A H(U,Y)>\lambda\}
        \right|
        \le
        C_{a,b}
        \int_{\{|H|>\lambda/2\}}
        \Phi\left(\frac{2|H(U,Y)|}{\lambda}\right)dU\,dY .
$$
Using the distribution formula for the \(L^p\) norm and then Tonelli's theorem,
$$
\begin{aligned}
        \|\mathcal A H\|_{L^p(\mathbb R\times[a,b])}^p
        &=
        p\int_0^\infty
        \lambda^{p-1}
        \left|
        \{(U,Y):a\le Y\le b,
        \mathcal A H(U,Y)>\lambda\}
        \right|d\lambda                                      \\
        &\le
        C_{p,a,b}
        \int_{\mathbb R^2}
        \int_0^{2|H(U,Y)|}
        \lambda^{p-1}
        \Phi\left(\frac{2|H(U,Y)|}{\lambda}\right)d\lambda
        \,dU\,dY .
\end{aligned}
$$
If \(A=|H(U,Y)|\), then for \(A>0\) the inner integral is, after the change of
variables \(\lambda=A t\),
$$
        A^p
        \int_0^2
        t^{p-1}\Phi\left(\frac2t\right)dt .
$$
The last integral is finite because, as \(t\downarrow0\),
$$
        t^{p-1}\Phi\left(\frac2t\right)
        \lesssim
        t^{p-2}\log\log\left(e^e+\frac2t\right),
$$
and \(p>1\).  Therefore
$$
        \|\mathcal A H\|_{L^p(\mathbb R\times[a,b])}^p
        \le
        C_{p,a,b}\int_{\mathbb R^2}|H(U,Y)|^p\,dU\,dY .
$$
Taking \(p\)-th roots proves the theorem.
\end{proof}

Now we introduce the elementary estimates for \(\Phi(s)=s\log\log(e^e+s)\).

\begin{lemma}
\label{lem:Phi-basic}
Let
$
        \Phi(s)=s\log\log(e^e+s),
        \ s\ge0.
$
Then
\begin{equation}\label{eq:Phi-basic}
        s\le \Phi(s),
        \qquad
        \Phi(2s)\le C\Phi(s),
        \qquad s\ge0,
\end{equation}
where \(C>0\) is an absolute constant.
\end{lemma}

\begin{proof}
Since
\(
        e^e+s\ge e^e,
\)
we have
\(
        \log(e^e+s)\ge e,
\)
and therefore
\(
        \log\log(e^e+s)\ge 1.
\)
Hence
$$
        \Phi(s)=s\log\log(e^e+s)\ge s.
$$

We now prove the doubling estimate.  Put
\(
        A(s)=\log(e^e+s).
\)
Then \(A(s)\ge e\).  Also,
$$
        e^e+2s
        \le
        2(e^e+s),
$$
so
$$
        \log(e^e+2s)
        \le
        \log 2+\log(e^e+s)
        =
        \log 2+A(s).
$$
Since \(A(s)\ge e\), we have
$$
        \log 2+A(s)
        \le
        \left(1+\frac{\log 2}{e}\right)A(s).
$$
Taking logarithms gives
$$
        \log\log(e^e+2s)
        \le
        \log A(s)+
        \log\left(1+\frac{\log 2}{e}\right).
$$
Because \(\log A(s)\ge1\), this implies
\(
        \log\log(e^e+2s)
        \le
        C\log\log(e^e+s).
\)
Therefore
$$
        \Phi(2s)
        =
        2s\log\log(e^e+2s)
        \le
        C s\log\log(e^e+s)
        =
        C\Phi(s).
$$
This proves \eqref{eq:Phi-basic}.
\end{proof}

\begin{proposition}
\label{prop:geo-slice-Lloglog}
For every measurable \(F\) and every \(\lambda>0\),
\begin{equation}\label{eq:geo-slice-Lloglog}
        \nu_n\{(u,y):\mathfrak M_{\mathrm{geo}}F(u,y)>\lambda\}
        \le
        C_n
        \int_{\mathbb R\times\mathbb R_+}
        \Phi\left(\frac{|F(u,y)|}{\lambda}\right)
        d\nu_n(u,y).
\end{equation}
\end{proposition}

\begin{proof}
Choose a non-negative smooth dyadic partition of unity
$$
        1=\sum_{k\in\mathbb Z}\chi_k(y),
        \qquad
        \chi_k(y)=\chi(2^{-k}y),
$$
where \(\chi\in C_c^\infty((0,\infty))\).  Fix constants
\(0<a<b<\infty\) such that
\(
        \operatorname{supp}\chi\subset [a,b],
\) e.g. take $a=\frac12$ and $b=2$.
Thus
$$
        \operatorname{supp}\chi_k\subset [a2^k,b2^k].
$$
Write
$$
        F=\sum_{k\in\mathbb Z}F_k,
        \qquad
        F_k(u,y)=F(u,y)\chi_k(y).
$$

For each height \(y>0\), define the active index set
$$
        I(y)
        =
        \left\{
        k\in\mathbb Z:
        [y\operatorname{sech}1,y]\cap [a2^k,b2^k]\neq\varnothing
        \right\}.
$$
If \(k\in I(y)\), then there exists \(s\) such that
$
        s\in [y\operatorname{sech}1,y]\cap [a2^k,b2^k].
$
Hence
$
        a2^k\le s\le y
$
and
$
        y\operatorname{sech}1\le s\le b2^k.
$
Therefore
$$
        \frac{y\operatorname{sech}1}{b}
        \le
        2^k
        \le
        \frac{y}{a}.
$$
Taking base-two logarithms gives
$$
        \log_2 y+\log_2\left(\frac{\operatorname{sech}1}{b}\right)
        \le
        k
        \le
        \log_2 y+\log_2\left(\frac1a\right).
$$
The length of this interval is
$$
        \log_2\left(\frac{b}{a\operatorname{sech}1}\right)
        =
        \log_2\left(\frac{b\cosh 1}{a}\right),
$$
which is independent of \(y\).  Hence there is a constant
$$
        N_0
        =
        \left\lceil
        \log_2\left(\frac{b\cosh 1}{a}\right)
        \right\rceil+2
$$
such that
$
        |I(y)|\le N_0
        \ \text{for every }y>0.
$

Therefore, pointwise,
\begin{equation}\label{eq:geo-active-pointwise}
        \mathfrak M_{\mathrm{geo}}F(u,y)
        \le
        \sum_{k\in I(y)}
        \mathfrak M_{\mathrm{geo}}F_k(u,y).
\end{equation}

We next prove the uniform localized dyadic estimate
\begin{equation}\label{eq:geo-dyadic-localized-Lloglog}
        \nu_n\left(
        \left\{
        (u,y):
        \mathfrak M_{\mathrm{geo}}F_k(u,y)>\mu
        \right\}
        \right)
        \le
        C_n
        \int
        \Phi\left(\frac{|F_k(u,y)|}{\mu}\right)
        d\nu_n(u,y),
\end{equation}
with a constant independent of \(k\) and \(\mu>0\).

Fix \(k\).  On the support of \(F_k\), the height satisfies
\(
        a2^k\le y\le b2^k.
\)
If
$
        \mathfrak M_{\mathrm{geo}}F_k(u,y)\neq0,
$
then for some \(0\le t\le1\),
$
        y\operatorname{sech}t\in [a2^k,b2^k].
$
Consequently the height is also localized:
\(
        a2^k\le y\le b(\cosh 1)2^k.
\)

Make the normalized change of variables
$
        u=2^kU,
        \ 
        y=2^kY,
$
and set
$
        \widetilde F_k(U,Y)=F_k(2^kU,2^kY).
$
Then
$$
        d\nu_n(u,y)
        =
        \frac{du\,dy}{y^{n+1}}
        =
        2^{k(1-n)}
        \frac{dU\,dY}{Y^{n+1}}.
$$

For every \((U,Y)\), we have the exact identity
$$
        \mathfrak M_{\mathrm{geo}}F_k(2^kU,2^kY)
        =
        \sup_{0<r\le1}
        \frac1r\int_0^r
        |\widetilde F_k(U+Y\tanh t,Y\operatorname{sech}t)|\,dt .
$$

The support of \(\widetilde F_k\) in the second variable lies in $[a,b]$, 
and hence the support of $ \mathfrak M_{\mathrm{geo}}F_k$ in $Y$ lies in
$$
        a\le Y\le b\cosh 1.
$$
On this fixed region, \(Y^{-(n+1)}dU\,dY\) is comparable to Lebesgue measure
\(dU\,dY\), with constants depending only on \(n,a,b\).

Define
$$
        \widetilde E_k(\mu)
        =
        \left\{
        (U,Y):
        a\le Y\le b\cosh 1,\ 
        \sup_{0<r\le1}
        \frac1r\int_0^r
        |\widetilde F_k(U+Y\tanh t,Y\operatorname{sech}t)|\,dt
        >\mu
        \right\}.
$$

By Theorem~\ref{thm:STW-local-finite-type}, applied with upper strip bound
\(b\cosh 1\), we have
\begin{align*}\label{eq:normalized-dyadic-set}
        |\widetilde E_k(\mu)|
        &\le
        C_{a,b}
        \int_{\mathbb R\times\mathbb R_+}
        \Phi\left(\frac{|\widetilde F_k(U,Y)|}{\mu}\right)
        dU\,dY\\
        &\le
        C_{n,a,b}
        \int_{\mathbb R\times\mathbb R_+}
        \Phi\left(\frac{|\widetilde F_k(U,Y)|}{\mu}\right)
        Y^{-(n+1)}\,dU\,dY .
\end{align*}

For this application, the upper strip bound is \(B=b\cosh 1\).  Thus the
patching in Theorem~\ref{thm:STW-local-finite-type} uses \(U\)-intervals of
length \(8B\), enlarged by \(B\) on both sides, and the enlarged intervals
have overlap at most \(2\).  Since \(B\) is fixed, both the chart size and the
overlap bound are independent of \(k\).

Let
$$
        E_k(\mu)
        =
        \{(2^kU,2^kY):(U,Y)\in \widetilde E_k(\mu)\}.
$$
Then
$$
\begin{aligned}
        \nu_n(E_k(\mu))
        &=
        2^{k(1-n)}
        \int_{\widetilde E_k(\mu)}
        Y^{-(n+1)}\,dU\,dY                                      \\
        &\le
        C_{n,a,b}
        2^{k(1-n)}
        |\widetilde E_k(\mu)|                                    \\
        &\le
        C_{n,a,b}
        2^{k(1-n)}
        \int
        \Phi\left(\frac{|\widetilde F_k(U,Y)|}{\mu}\right)
        Y^{-(n+1)}\,dU\,dY                                      \\
        &=
        C_{n,a,b}
        \int
        \Phi\left(\frac{|F_k(u,y)|}{\mu}\right)
        d\nu_n(u,y).
\end{aligned}
$$
This proves \eqref{eq:geo-dyadic-localized-Lloglog}.

We now perform the final summation using the restricted level sets.  Define
$$
        \mathcal E_k
        =
        \left\{
        (u,y):
        k\in I(y),\ 
        \mathfrak M_{\mathrm{geo}}F_k(u,y)>\lambda/N_0
        \right\}.
$$
By construction, if \((u,y)\in\mathcal E_k\), then
\(
        a2^k\le y\le b(\cosh 1)2^k.
\)

Thus each \(\mathcal E_k\) is localized to the region \(y\simeq 2^k\).

From \eqref{eq:geo-active-pointwise}, if
$
        \mathfrak M_{\mathrm{geo}}F(u,y)>\lambda,
$
then
$$
        \sum_{k\in I(y)}
        \mathfrak M_{\mathrm{geo}}F_k(u,y)>\lambda.
$$
Since \(|I(y)|\le N_0\) uniformly in \(y\), there exists some \(k\in I(y)\) such that
$
        \mathfrak M_{\mathrm{geo}}F_k(u,y)>\lambda/N_0.
$
Therefore
$$
        \{(u,y):\mathfrak M_{\mathrm{geo}}F(u,y)>\lambda\}
        \subset
        \bigcup_{k\in\mathbb Z}\mathcal E_k.
$$
Hence
$$
        \nu_n\{(u,y):\mathfrak M_{\mathrm{geo}}F(u,y)>\lambda\}
        \le
        \sum_{k\in\mathbb Z}\nu_n(\mathcal E_k).
$$
Using \eqref{eq:geo-dyadic-localized-Lloglog} with \(\mu=\lambda/N_0\), we get
$$
\begin{aligned}
        \nu_n\{(u,y):\mathfrak M_{\mathrm{geo}}F(u,y)>\lambda\}
        &\le
        C_n
        \sum_{k\in\mathbb Z}
        \int
        \Phi\left(\frac{N_0|F_k(u,y)|}{\lambda}\right)
        d\nu_n(u,y).
\end{aligned}
$$

Let
$$
        N_\chi
        =
        \sup_{y>0}\#\{k:\chi_k(y)\neq0\}.
$$
This number is finite because \(\operatorname{supp}\chi\subset[a,b]\).  Since
the partition is non-negative and \(\sum_k\chi_k=1\), we have
\(0\le\chi_k\le1\), and hence \(|F_k|\le |F|\).  Therefore, by the doubling
property of \(\Phi\),
$$
        \Phi\left(\frac{N_0|F_k(u,y)|}{\lambda}\right)
        \le
        C_{N_0}
        \Phi\left(\frac{|F(u,y)|}{\lambda}\right).
$$
Summing over at most \(N_\chi\) indices gives
$$
        \sum_k
        \Phi\left(\frac{N_0|F_k(u,y)|}{\lambda}\right)
        \le
        C_{N_0,\chi}
        \Phi\left(\frac{|F(u,y)|}{\lambda}\right).
$$

Therefore
$$
        \nu_n\{(u,y):\mathfrak M_{\mathrm{geo}}F(u,y)>\lambda\}
        \le
        C_n
        \int
        \Phi\left(\frac{|F(u,y)|}{\lambda}\right)
        d\nu_n(u,y).
$$
This proves \eqref{eq:geo-slice-Lloglog}.
\end{proof}

\begin{proof}[Proof of Theorem \ref{thm:geo-local-Lloglog}]

Recall that 
$$
        \mathcal M_{\gamma_\omega}^{\mathrm{loc}}f(u\omega+z,y)
        =
        \mathfrak M_{\mathrm{geo}}(f_z)(u,y)
$$
holds for each fixed \(z\), with \(f_z(u,y)=f(u\omega+z,y)\).

Therefore, by Proposition~\ref{prop:geo-slice-Lloglog} and Fubini,
$$
\begin{aligned}
        \mu\left\{(x,y):
        \mathcal M_{\gamma_\omega}^{\mathrm{loc}}f(x,y)>\lambda
        \right\}
        &=
        \int_{\omega^\perp}
        \nu_n\left\{(u,y):
        \mathfrak M_{\mathrm{geo}}(f_z)(u,y)>\lambda
        \right\}\,dz                                      \\
        &\le
        C_n
        \int_{\omega^\perp}
        \int_{\mathbb R\times\mathbb R_+}
        \Phi\left(\frac{|f_z(u,y)|}{\lambda}\right)
        d\nu_n(u,y)\,dz                                  \\
        &=
        C_n
        \int_{G_n}
        \Phi\left(\frac{|f(x,y)|}{\lambda}\right)
        d\mu(x,y).
\end{aligned}
$$
This proves the theorem.
\end{proof}

The same slicing argument gives the \(L^p\) boundedness of \(\mathcal M_{\gamma_\omega}^{\mathrm{loc}}\) from the \(L^p\) boundedness of \(\mathfrak M_{\mathrm{geo}}\).

\begin{proposition}
\label{prop:geo-local-p}
For every fixed \(\omega\in S^{n-1}\) and every \(1<p\le\infty\),
$$
        \|\mathcal M_{\gamma_\omega}^{\mathrm{loc}}f\|_{L^p(G_n)}
        \le
        C_{p,n}\|f\|_{L^p(G_n)} .
$$
\end{proposition}

\begin{proof}
The case \(p=\infty\) is immediate:
\(
        \mathcal M_{\gamma_\omega}^{\mathrm{loc}}f(x,y)
        \le
        \|f\|_{L^\infty(G_n)} .
\)

Let \(1<p<\infty\).  We first prove the corresponding slice estimate.  Recall
$$
        d\nu_n(u,y)=\frac{du\,dy}{y^{n+1}}
$$
and
$$
        \mathfrak M_{\mathrm{geo}}F(u,y)
        =
        \sup_{0<r\le1}
        \frac1r\int_0^r
        |F(u+y\tanh t,y\operatorname{sech}t)|\,dt .
$$
We claim that
\begin{equation}\label{eq:slice-geo-local-p}
        \|\mathfrak M_{\mathrm{geo}}F\|_{L^p(d\nu_n)}
        \le
        C_{p,n}\|F\|_{L^p(d\nu_n)} .
\end{equation}

Similar to the proof of Proposition \ref{prop:geo-slice-Lloglog}, we define
$$
        I(y)
        =
        \left\{
        k\in\mathbb Z:
        [y\operatorname{sech}1,y]\cap [a2^k,b2^k]\neq\varnothing
        \right\}.
$$
Then it holds
$$
        |I(y)|\le N_0
        \qquad \text{for every } y>0
$$ 
and
$$
        \mathfrak M_{\mathrm{geo}}F(u,y)
        \le
        \sum_{k\in I(y)}
        \mathfrak M_{\mathrm{geo}}F_k(u,y).
$$

Thus, it suffices to prove the uniform dyadic estimate
\begin{equation}\label{eq:dyadic-slice-p}
        \|\mathfrak M_{\mathrm{geo}}F_k\|_{L^p(d\nu_n)}
        \le
        C_{p,n}\|F_k\|_{L^p(d\nu_n)},
        \qquad k\in\mathbb Z.
\end{equation}
Indeed, if \(\mathfrak M_{\mathrm{geo}}F_k(u,y)\neq0\), then for some
\(0\le t\le1\),
\(
        y\operatorname{sech}t\in [a2^k,b2^k].
\)

Hence
$$
        a2^k\le y\le b(\cosh 1)2^k.
$$
Thus the variable is restricted to \(y\simeq 2^k\).  Make the normalized
change of variables
$$
        u=2^kU,
        \qquad
        y=2^kY.
$$
Then \(Y\) ranges in a fixed compact interval $[a,b\cosh 1]$, independent
of \(k\), and
$$
        d\nu_n(u,y)
        =
        2^{k(1-n)}\frac{dU\,dY}{Y^{n+1}}.
$$
On this fixed \(Y\)-range, the measure \(Y^{-(n+1)}dU\,dY\) is comparable to
Lebesgue measure.

In the normalized variables \(u=2^kU\), \(y=2^kY\), this dyadic piece is the
same normalized geodesic averaging operator that appeared in
Theorem~\ref{thm:STW-local-finite-type}.  The finite-type geometry on each
fixed \(Y\)-strip has already been verified there: after normalization the
curve has uniformly bounded \(C^N\)-seminorms and satisfies a uniform
order-two curvature condition.  Therefore the local finite-type
\(L^p\) estimate in Theorem~\ref{thm:STW-local-finite-type-Lp} gives
\eqref{eq:dyadic-slice-p}, with a constant independent of the dyadic index
\(k\).

The passage from bounded \(U\)-patches to the full \(U\)-axis is also uniform.
The normalized \(U\)-displacement is bounded on the fixed \(Y\)-strip, so a
bounded-overlap partition of the \(U\)-axis gives the same estimate globally in
\(U\).  Therefore the constant in \eqref{eq:dyadic-slice-p} depends only on
\(p,n\) and on the fixed strip constants, not on \(k\).

Now use the finite-overlap pointwise inequality.  Since \(|I(y)|\le N_0\),
$$
\begin{aligned}
        \|\mathfrak M_{\mathrm{geo}}F\|_{L^p(d\nu_n)}^p
        &\le
        \int
        \left(
        \sum_{k\in I(y)}
        \mathfrak M_{\mathrm{geo}}F_k(u,y)
        \right)^p
        d\nu_n(u,y)                                      \\
        &\le
        N_0^{p-1}
        \sum_{k\in\mathbb Z}
        \|\mathfrak M_{\mathrm{geo}}F_k\|_{L^p(d\nu_n)}^p \\
        &\le
        C_{p,n}
        \sum_{k\in\mathbb Z}
        \|F_k\|_{L^p(d\nu_n)}^p                           \\
        &\le
        C_{p,n}
        \|F\|_{L^p(d\nu_n)}^p,
\end{aligned}
$$
because the functions \(\chi_k\) have bounded overlap.  This proves
\eqref{eq:slice-geo-local-p}.

Recall that for each fixed \(z\in\omega^\perp\), we have defined
$$
        f_z(u,y)=f(u\omega+z,y).
$$
Since
$$
        x=u\omega+z,
        \qquad
        d\mu(x,y)=dz\,d\nu_n(u,y),
        \qquad
        d\nu_n(u,y)=\frac{du\,dy}{y^{n+1}},
$$
Fubini's theorem implies that \(f_z\in L^p(d\nu_n)\) for almost every
\(z\in\omega^\perp\), and
$$
        \int_{\omega^\perp}
        \|f_z\|_{L^p(d\nu_n)}^p\,dz
        =
        \|f\|_{L^p(G_n)}^p.
$$

The local fixed-geodesic operator leaves the transverse variable \(z\) fixed, that is
$$
        \mathcal M_{\gamma_\omega}^{\mathrm{loc}}f(u\omega+z,y)
        =
        \mathfrak M_{\mathrm{geo}}(f_z)(u,y).
$$

The constant in the slice estimate is uniform in the transverse parameter
\(z\).  Indeed, the estimate \eqref{eq:slice-geo-local-p} is an operator
norm estimate on the fixed measure space
$$
        \bigl(\mathbb R\times\mathbb R_+,d\nu_n\bigr),
        \qquad
        d\nu_n(u,y)=\frac{du\,dy}{y^{n+1}}.
$$
It says that for every \(F\in L^p(d\nu_n)\),
$$
        \|\mathfrak M_{\mathrm{geo}}F\|_{L^p(d\nu_n)}
        \le
        C_{p,n}\|F\|_{L^p(d\nu_n)},
$$
where \(C_{p,n}\) depends only on \(p\), \(n\), and the local truncation, but
not on any additional parameter.  Since neither \(\mathfrak M_{\mathrm{geo}}\)
nor \(d\nu_n\) depends on \(z\), this same constant applies to every slice
\(f_z\).

Applying the slice estimate \eqref{eq:slice-geo-local-p} to \(f_z\), for
almost every \(z\), and then integrating in \(z\), gives
$$
\begin{aligned}
        \|\mathcal M_{\gamma_\omega}^{\mathrm{loc}}f\|_{L^p(G_n)}^p
        =
        \int_{\omega^\perp}
        \|\mathfrak M_{\mathrm{geo}}(f_z)\|_{L^p(d\nu_n)}^p\,dz        
        \le
        C_{p,n}^p
        \int_{\omega^\perp}
        \|f_z\|_{L^p(d\nu_n)}^p\,dz                                  
        =
        C_{p,n}^p
        \|f\|_{L^p(G_n)}^p .
\end{aligned}
$$
Taking \(p\)-th roots yields
$$
        \|\mathcal M_{\gamma_\omega}^{\mathrm{loc}}f\|_{L^p(G_n)}
        \le
        C_{p,n}\|f\|_{L^p(G_n)}.
$$
This proves the proposition.
\end{proof}

\subsection{The large-time tail}

\begin{proposition}
\label{prop:geo-tail-L1}
For every fixed \(\omega\in S^{n-1}\),
$$
        \|\mathcal M_{\gamma_\omega}^{\mathrm{tail}}f\|_{L^1(G_n)}
        \le
        C_n\|f\|_{L^1(G_n)}.
$$
Consequently,
$$
        \|\mathcal M_{\gamma_\omega}^{\mathrm{tail}}f\|_{L^{1,\infty}(G_n)}
        \le
        C_n\|f\|_{L^1(G_n)}.
$$
\end{proposition}

\begin{proof}
For \(r\ge1\),
$$
        \frac1r\int_0^r |S_tf(x,y)|\,dt
        \le
        \int_0^\infty |S_tf(x,y)|\,dt.
$$
Therefore
$$
        \mathcal M_{\gamma_\omega}^{\mathrm{tail}}f(x,y)
        \le
        \int_0^\infty |S_tf(x,y)|\,dt.
$$

From Lemma~\ref{lem:geodesic-norm} we have
$$
        \|S_tf\|_{L^1(G_n)}
        =
        (\operatorname{sech}t)^n
        \|f\|_{L^1(G_n)}.
$$
Hence
$$
\begin{aligned}
        \|\mathcal M_{\gamma_\omega}^{\mathrm{tail}}f\|_{L^1(G_n)}
        &\le
        \int_0^\infty
        \|S_tf\|_{L^1(G_n)}\,dt                         
        =
        \left(
        \int_0^\infty(\operatorname{sech}t)^n\,dt
        \right)
        \|f\|_{L^1(G_n)}                                
        \le
        C_n\|f\|_{L^1(G_n)}.
\end{aligned}
$$
The weak \((1,1)\) estimate follows from Chebyshev's inequality.
\end{proof}

For \(1<p<\infty\), the same argument gives a strong tail estimate.  Since
$$
        \mathcal M_{\gamma_\omega}^{\mathrm{tail}}f(x,y)
        \le
        \int_0^\infty |S_tf(x,y)|\,dt,
$$
Minkowski's integral inequality and Lemma~\ref{lem:geodesic-norm} give
$$
\begin{aligned}
        \|\mathcal M_{\gamma_\omega}^{\mathrm{tail}}f\|_{L^p(G_n)}
        &\le
        \int_0^\infty \|S_tf\|_{L^p(G_n)}\,dt                 
        =
        \left(\int_0^\infty(\operatorname{sech}t)^{n/p}\,dt\right)
        \|f\|_{L^p(G_n)}.
\end{aligned}
$$
The integral is finite because \(\operatorname{sech}t\simeq 2e^{-t}\) as
\(t\to\infty\).  Thus
\begin{equation}\label{eq:geo-tail-p}
        \|\mathcal M_{\gamma_\omega}^{\mathrm{tail}}f\|_{L^p(G_n)}
        \le
        C_{p,n}\|f\|_{L^p(G_n)},
        \qquad 1<p<\infty .
\end{equation}
For \(p=\infty\), the estimate follows directly from
$$
        \frac1r\int_0^r |S_tf(x,y)|\,dt
        \le
        \|f\|_{L^\infty(G_n)}.
$$

\subsection{Fixed-geodesic bounds}

We now combine the local estimates with the large-time tail estimates.

\begin{theorem}
\label{thm:fixed-geodesic-Lloglog}
For every fixed \(\omega\in S^{n-1}\), the fixed-geodesic maximal operator
\(\mathcal M_{\gamma_\omega}\) is bounded on \(L^p(G_n)\) for every
\(1<p\le\infty\).  Moreover,
$$
        \mathcal M_{\gamma_\omega}
        :
        L\log\log L(G_n)
        \longrightarrow
        L^{1,\infty}(G_n)
$$
in the weak Orlicz sense: for every measurable \(f\) and every \(\lambda>0\),
\begin{equation}\label{eq:fixed-geo-Lloglog-endpoint}
        \mu\left\{(x,y):
        \mathcal M_{\gamma_\omega}f(x,y)>\lambda
        \right\}
        \le
        C_n
        \int_{G_n}
        \Phi\left(\frac{|f(x,y)|}{\lambda}\right)
        d\mu(x,y).
\end{equation}
\end{theorem}

\begin{proof}
We first prove the endpoint estimate.  By definition,
$$
        \mathcal M_{\gamma_\omega}f
        \le
        \mathcal M_{\gamma_\omega}^{\mathrm{loc}}f
        +
        \mathcal M_{\gamma_\omega}^{\mathrm{tail}}f .
$$
Therefore
$$
\begin{aligned}
        \mu\left\{
        \mathcal M_{\gamma_\omega}f>\lambda
        \right\}
        &\le
        \mu\left\{
        \mathcal M_{\gamma_\omega}^{\mathrm{loc}}f>\lambda/2
        \right\}                                             +
        \mu\left\{
        \mathcal M_{\gamma_\omega}^{\mathrm{tail}}f>\lambda/2
        \right\}.
\end{aligned}
$$

The local term is controlled by the local fixed-geodesic
\(L\log\log L\) endpoint estimate.  Namely, by
Theorem~\ref{thm:geo-local-Lloglog}, applied with \(\lambda/2\) in place of
\(\lambda\),
$$
        \mu\left\{
        \mathcal M_{\gamma_\omega}^{\mathrm{loc}}f>\lambda/2
        \right\}
        \le
        C_n
        \int_{G_n}
        \Phi\left(\frac{2|f|}{\lambda}\right)
        d\mu .
$$
By the doubling estimate for \(\Phi\), recorded in \eqref{eq:Phi-basic},
$$
        \Phi(2s)\le C\Phi(s),
        \qquad s\ge0,
$$
and hence
\begin{equation}\label{eq:fixed-geo-local-endpoint-term}
        \mu\left\{
        \mathcal M_{\gamma_\omega}^{\mathrm{loc}}f>\lambda/2
        \right\}
        \le
        C_n
        \int_{G_n}
        \Phi\left(\frac{|f|}{\lambda}\right)
        d\mu .
\end{equation}

For the tail term, Proposition~\ref{prop:geo-tail-L1} gives the weak
\(L^1\) estimate
$$
\begin{aligned}
        \mu\left\{
        \mathcal M_{\gamma_\omega}^{\mathrm{tail}}f>\lambda/2
        \right\}
         \le
        \frac{C_n}{\lambda}
        \|f\|_{L^1(G_n)} .
\end{aligned}
$$
Since \(\Phi(s)\ge s\) for every \(s\ge0\), we have
$$
        \frac1\lambda\|f\|_{L^1(G_n)}
        =
        \int_{G_n}\frac{|f|}{\lambda}\,d\mu
        \le
        \int_{G_n}
        \Phi\left(\frac{|f|}{\lambda}\right)d\mu .
$$
Thus
\begin{equation}\label{eq:fixed-geo-tail-endpoint-term}
        \mu\left\{
        \mathcal M_{\gamma_\omega}^{\mathrm{tail}}f>\lambda/2
        \right\}
        \le
        C_n
        \int_{G_n}
        \Phi\left(\frac{|f|}{\lambda}\right)d\mu .
\end{equation}
Combining \eqref{eq:fixed-geo-local-endpoint-term} and
\eqref{eq:fixed-geo-tail-endpoint-term} proves
\eqref{eq:fixed-geo-Lloglog-endpoint}.

We now prove the strong \(L^p\) estimates.  Let \(1<p<\infty\).  By
Proposition~\ref{prop:geo-local-p},
$$
        \|\mathcal M_{\gamma_\omega}^{\mathrm{loc}}f\|_{L^p(G_n)}
        \le
        C_{p,n}\|f\|_{L^p(G_n)} .
$$
For the large-time part, \eqref{eq:geo-tail-p} gives
$$
        \|\mathcal M_{\gamma_\omega}^{\mathrm{tail}}f\|_{L^p(G_n)}
        \le
        C_{p,n}\|f\|_{L^p(G_n)} .
$$
Since
$
        \mathcal M_{\gamma_\omega}f
        \le
        \mathcal M_{\gamma_\omega}^{\mathrm{loc}}f
        +
        \mathcal M_{\gamma_\omega}^{\mathrm{tail}}f,
$
we obtain
$$
        \|\mathcal M_{\gamma_\omega}f\|_{L^p(G_n)}
        \le
        C_{p,n}\|f\|_{L^p(G_n)} .
$$
For \(p=\infty\), the estimate follows directly from
$$
        \mathcal M_{\gamma_\omega}f(x,y)
        \le
        \|f\|_{L^\infty(G_n)} .
$$
This proves the strong \(L^p\) estimates.
\end{proof}

\subsection{\texorpdfstring{Failure of strong $L^1$ endpoint}{Failure of strong L1 endpoint}}

The strong $L^1$ endpoint fails.  This is a local obstruction and is already present at small times.

\begin{proposition}\label{prop:geo_not_l1}
For every fixed $\omega\in S^{n-1}$, the operator $\mathcal M_{\gamma_\omega}$ is not bounded from $L^1(G_n)$ to $L^1(G_n)$.
\end{proposition}

\begin{proof}
By rotation invariance in the $x$ variable it is enough to consider $\omega=e_1$.  Write $x=(x_1,x')$, with $x'\in\mathbb R^{n-1}$.  Let $0<\varepsilon<10^{-2}$ and set
$$
        f_\varepsilon(x,y)
        =
        \mathbf 1_{[0,\varepsilon]}(x_1)
        \mathbf 1_{[0,1]^{n-1}}(x')
        \mathbf 1_{[1,2]}(y).
$$
Then
$$
        \|f_\varepsilon\|_{L^1(G_n)}
        =
        \int_{[0,\varepsilon]}dx_1\int_{[0,1]^{n-1}}dx'
        \int_1^2\frac{dy}{y^{n+1}}
        \simeq_n \varepsilon.
$$

We now obtain a lower bound for the maximal function.  Take
$$
        -\frac14<x_1<-2\varepsilon,
        \qquad
        x'\in[0,1]^{n-1},
        \qquad
        \frac32<y<\frac74.
$$

We claim that the set of \(t\) for which
$$
        0<x_1+y\tanh t<\varepsilon
$$
contains an interval centered of length at least
\(\varepsilon/2y\). 

Indeed, put
$$
        \tau=\operatorname{arctanh}\left(\frac{-x_1}{y}\right).
$$
For the above range of \((x_1,y)\), one has \(0<\tau<1/5\).  The map
\(t\mapsto x_1+y\tanh t\) is increasing, and its derivative is
\(y\operatorname{sech}^2t\), which is less than $y$.  Thus the interval $(\tau,\tau+\varepsilon/2y)$ is what we want.

In this interval we also have
$$
        1<y\operatorname{sech}t<2.
$$
Therefore \(f_\varepsilon(x+y e_1\tanh t,y\operatorname{sech}t)=1\) on a
\(t\)-interval of length at least \(\varepsilon/2y\).

Choose \(r=2\tau\).  Since \(|x_1|>2\varepsilon\) and
\(y\in(3/2,7/4)\), the interval just described is contained in \([0,r]\), and
\(r\simeq |x_1|\).  Hence
$$
        \mathcal M_{\gamma_{e_1}}f_\varepsilon(x,y)
        \ge
        \frac{c\varepsilon/y}{r}
        \ge
        c_n\frac{\varepsilon}{|x_1|}.
$$
It follows that
\begin{align*}
        \|\mathcal M_{\gamma_{e_1}}f_\varepsilon\|_{L^1(G_n)}
        &\ge
        c_n
        \int_{3/2}^{7/4}\frac{dy}{y^{n+1}}
        \int_{[0,1]^{n-1}}dx'
        \int_{-1/4}^{-2\varepsilon}\frac{\varepsilon}{|x_1|}\,dx_1 
        \ge
        c_n\varepsilon\log\frac1\varepsilon.
\end{align*}
Since $\|f_\varepsilon\|_{L^1(G_n)}\simeq_n\varepsilon$, the ratio
$$
        \frac{\|\mathcal M_{\gamma_{e_1}}f_\varepsilon\|_{L^1(G_n)}}
        {\|f_\varepsilon\|_{L^1(G_n)}}
$$
tends to infinity as $\varepsilon\downarrow0$.  Thus no strong $L^1$ bound is possible.
\end{proof}

\section{Probabilistic Perspectives}\label{sec:prob}

This section records two standard consequences of the modular calculation.  The first is the vertical drift of Brownian motion in logarithmic coordinates.  The second is a discrete random-walk maximal estimate for right convolution powers.  The notation in this section is as follows: $m$ denotes left Haar measure on $G_n$, while $\sigma$ denotes a compactly supported probability measure on $G_n$.

\subsection{Brownian motion and logarithmic drift}

Identify $G_n$ with the upper half-space $\mathbb R^n\times\mathbb R_+$ equipped with the hyperbolic metric
$$
        ds^2=y^{-2}(|dx|^2+dy^2).
$$
The Laplace--Beltrami operator is
\begin{equation}\label{eq:laplacian}
        \Delta_{G_n}
        =y^2\left(\sum_{j=1}^n\frac{\partial^2}{\partial x_j^2}
        +\frac{\partial^2}{\partial y^2}\right)
        -(n-1)y\frac{\partial}{\partial y}.
\end{equation}
Thus Brownian motion with generator $\frac12\Delta_{G_n}$ satisfies
\begin{equation}\label{eq:sde}
\begin{cases}
        dX_t^j=Y_t\,dW_t^j,\qquad j=1,\ldots,n,\\[4pt]
        dY_t=Y_t\,dW_t^{n+1}-\dfrac{n-1}{2}Y_t\,dt,
\end{cases}
\end{equation}
where $W_t=(W_t^1,\ldots,W_t^{n+1})$ is standard Brownian motion in $\mathbb R^{n+1}$.

Put $U_t=\log Y_t$.  By It\^o's formula,
\begin{align*}
        dU_t
        &=\frac1{Y_t}\,dY_t-\frac1{2Y_t^2}(dY_t)^2 
        =dW_t^{n+1}-\frac{n-1}{2}\,dt-\frac12\,dt 
        =dW_t^{n+1}-\frac n2\,dt.
\end{align*}
Thus $\log Y_t$ is a one-dimensional Brownian motion with drift $-n/2$.  This is the stochastic counterpart of the identity
$$
        d\mu(x,e^u)=e^{-nu}\,dx\,du.
$$
The same exponent $n$ appears in the modular factor for right translations and in the Haar-compatible dilation averages.

\subsection{Fixed geodesics}

The fixed-geodesic maximal operator studied in Section~\ref{sec:geo} is an operator on functions on $G_n$.  It is not the non-tangential maximal function for harmonic functions on the upper half-space.  The connection is geometric only: both involve geodesics approaching the boundary.

For fixed $\omega$, the right translation
$$
        (x,y)\mapsto(x+y\omega\tanh t,y\operatorname{sech}t)
$$
changes $L^p(G_n)$ norm by the factor $(\operatorname{sech}t)^{n/p}$ for
$1\le p<\infty$.  This gives strong $L^1$ control of the large-time part of
the maximal operator.  The small-time part is parabolic: after slicing in the
orthogonal variables and normalizing near a point, it is modeled on the
parabolic curve
$$
        (U,V)\mapsto(U+t,V-\tfrac12t^2).
$$
Thus the endpoint used in this paper is the weak \(L\log\log L\) estimate
$$
        L\log\log L(G_n)
        \longrightarrow
        L^{1,\infty}(G_n),
$$
not a claimed full local weak $(1,1)$ estimate on all of $L^1(G_n)$.  The
counterexample in Proposition~\ref{prop:geo_not_l1} shows that the strong
$L^1$ endpoint fails logarithmically and locally.

\subsection{A random-walk maximal inequality}

Let $\sigma$ be a compactly supported probability measure on $G_n$.  Define the right convolution operator
$$
        R_\sigma f(g)=\int_{G_n}f(gh)\,d\sigma(h)
$$
and the discrete maximal operator
$$
        \widetilde M_\sigma f(g)
        =\sup_{N\ge1}\frac1N\sum_{k=1}^N |R_\sigma^k f(g)|.
$$
For $1\le p<\infty$ set
$$
        \rho_p(\sigma)=\int_{G_n}y(h)^{n/p}\,d\sigma(h).
$$

\begin{theorem}\label{thm:rw}
Let $1\le p<\infty$.  If
$$
        \rho_p(\sigma)<1,
$$
then $\widetilde M_\sigma$ is bounded on $L^p(G_n)$.
\end{theorem}

\begin{proof}
For fixed $h=(a,b)\in G_n$, \eqref{eq:right-translation-integral} gives
$$
        \int_{G_n}|f(gh)|^p\,dm(g)
        =b^n\int_{G_n}|f(g)|^p\,dm(g).
$$
Therefore
$$
        \|f(\cdot h)\|_{L^p(G_n)}=b^{n/p}\|f\|_{L^p(G_n)}.
$$
Minkowski's integral inequality gives
\begin{align*}
        \|R_\sigma f\|_{L^p(G_n)}
        &\le \int_{G_n}\|f(\cdot h)\|_{L^p(G_n)}\,d\sigma(h)
        =\rho_p(\sigma)\|f\|_{L^p(G_n)}.
\end{align*}
Thus $\|R_\sigma\|_{p\to p}\le\rho_p(\sigma)<1$.

By definition,
$$
        \widetilde M_\sigma f
        \le \sum_{k=1}^\infty R_\sigma^k |f|.
$$
The series converges in operator norm on $L^p(G_n)$, and hence
$$
        \|\widetilde M_\sigma f\|_{L^p(G_n)}
        \le
        \sum_{k=1}^\infty \|R_\sigma\|_{p\to p}^k\|f\|_{L^p(G_n)}
        \le
        \frac{\rho_p(\sigma)}{1-\rho_p(\sigma)}\|f\|_{L^p(G_n)}.
$$
The proof is complete.
\end{proof}

The condition $\rho_p(\sigma)<1$ is only a sufficient condition for $p>1$.  It is useful because it is explicit and follows directly from the modular function.  At $p=1$, the opposite drift gives a direct obstruction.

\begin{lemma}\label{lem:weak-l1-finite-support}
Let $h$ be supported in a set $S$ of finite measure.  If
$$
        \|h\|_{L^{1,\infty}}\le B,
        \qquad
        \|h\|_{L^\infty}\le M,
$$
then
$$
        \|h\|_{L^1}
        \le
        B\left(1+\log^+\frac{m(S)M}{B}\right).
$$
\end{lemma}

\begin{proof}
Using the distribution function formula,
$$
        \|h\|_{L^1}
        =\int_0^M m(\{|h|>\lambda\})\,d\lambda.
$$
Set $\lambda_0=B/m(S)$.  If $\lambda_0\ge M$, then $\|h\|_1\le m(S)M\le B$.  If $\lambda_0<M$, split the integral at $\lambda_0$.  For $0<\lambda<\lambda_0$ use $m(\{|h|>\lambda\})\le m(S)$, and for $\lambda_0<\lambda<M$ use $m(\{|h|>\lambda\})\le B/\lambda$.  This gives
$$
        \|h\|_1
        \le B+B\log\frac{M}{\lambda_0}
        =B\left(1+\log\frac{m(S)M}{B}\right).
$$
The proof is complete.
\end{proof}

\begin{proposition}\label{prop:rw-expansive}
Let $\sigma$ be a compactly supported probability measure on $G_n$.  If
$$
        \rho_1(\sigma)=\int_{G_n}y(h)^n\,d\sigma(h)>1,
$$
then $\widetilde M_\sigma$ is not bounded on $L^1(G_n)$ and is not of weak type $(1,1)$.
\end{proposition}

\begin{proof}
Choose a compact set $K\subset G_n$ with $0<m(K)<\infty$, and put $f=\mathbf 1_K$.  For $N\ge1$ define
$$
        A_N f=\frac1N\sum_{k=1}^N R_\sigma^k f.
$$
Since $f\ge0$, we have $0\le A_Nf\le\widetilde M_\sigma f$.

For non-negative $\phi$, \eqref{eq:right-translation-integral} implies
$$
        \|R_\sigma\phi\|_{L^1(G_n)}
        =\rho_1(\sigma)\|\phi\|_{L^1(G_n)}.
$$
Iterating gives
$$
        \|R_\sigma^k f\|_{L^1(G_n)}
        =\rho_1(\sigma)^k\|f\|_{L^1(G_n)}
$$
since all $R_\sigma^k f$ are non-negative for $k=0,1,2,\cdots$.

Therefore
\begin{equation}\label{eq:strong_growth}
        \|A_Nf\|_{L^1(G_n)}
        =\frac{\|f\|_{L^1(G_n)}}N\sum_{k=1}^N\rho_1(\sigma)^k
        \ge c\,\|f\|_{L^1(G_n)}\frac{\rho_1(\sigma)^N}{N},
\end{equation}
where $c>0$ is independent of $N$.  Since $A_Nf\le\widetilde M_\sigma f$, this rules out strong $L^1$ boundedness.

It remains to rule out weak type $(1,1)$.  Suppose that
$$
        \|\widetilde M_\sigma h\|_{L^{1,\infty}(G_n)}
        \le C\|h\|_{L^1(G_n)}.
$$
Then
$$
        \|A_Nf\|_{L^{1,\infty}(G_n)}
        \le C\|f\|_{L^1(G_n)}=:B
$$
for all $N$.  Also $\|A_Nf\|_\infty\le1$, because $\sigma$ is a probability measure and $0\le f\le1$.

Let $L=\operatorname{supp}\sigma$.  Since $K$ and $L$ are compact,
$$
        \operatorname{supp}(A_Nf)
        \subset \bigcup_{k=1}^N K(L^{-1})^k.
$$
We use the left-invariant Riemannian distance associated with
$$
        ds^2=y^{-2}(|dx|^2+dy^2).
$$
Choose \(R_K,R_L>0\) such that
$$
        K\subset B(e,R_K),
        \qquad
        L^{-1}\subset B(e,R_L).
$$
The triangle inequality and left invariance give
$$
        K(L^{-1})^k\subset B(e,R_K+kR_L),
$$
and therefore
$$
        \operatorname{supp}(A_Nf)
        \subset B(e,R_K+NR_L).
$$
For this metric, \(G_n\) is the upper half-space model of
\(\mathbb H^{n+1}\), and \(m\) is its Riemannian volume measure.  Thus
$$
        m(B(e,R))
        =|S^n|\int_0^R(\sinh r)^n\,dr
        \le C_ne^{nR},
        \qquad R\ge0.
$$
Consequently, there exist constants \(C_2,C_3>0\), depending only on
\(K,L\), and \(n\), such that
\begin{equation}\label{eq:supp_growth}
        m(\operatorname{supp}(A_Nf))\le C_2e^{C_3N}.
\end{equation}
Applying Lemma~\ref{lem:weak-l1-finite-support} to $h=A_Nf$, with $M=1$, gives
$$
        \|A_Nf\|_{L^1(G_n)}
        \le B\left(1+\log^+\frac{C_2e^{C_3N}}{B}\right)
        \le C'(1+N),
$$
where $C'$ is independent of $N$.  This contradicts the exponential lower bound \eqref{eq:strong_growth}.
\end{proof}

\bigskip
\bigskip
\bigskip

\noindent {\bf Acknowledgements}: The authors would like to thank the referees for their careful reading and valuable comments, which have helped us to improve the quality of this manuscript.

The first author is supported by the Australian Research Council through grant  DP260100485. 
The second author  is supported in part by NSTC through grant 111-2115-M-002-010-MY5. 
\bigskip

\bigskip
\bigskip

\noindent {\bf Conflicts of Interest}:
On behalf of all co-authors, the corresponding author states that there is no conflict of interest.

\bigskip
\bigskip

\noindent {\bf Data Availability Statement}:
Data sharing is not applicable to this article, as no datasets were generated or analysed during the current study.

\vspace{1cm}

(J. Li) School of Mathematical and Physical Sciences, Macquarie University, NSW, 2109, Australia.\\ 
{\it E-mail}: \texttt{ji.li@mq.edu.au}\\

(C.-Y. Shen) Department of Mathematics, National Taiwan University, Taiwan. \\ {\it E-mail}: \texttt{cyshen@math.ntu.edu.tw}\\

(C.-J. Wen) Department of Mathematics, Sun Yat-sen University, Guangzhou, 510275, China.
\\ 
{\it E-mail}: \texttt{wenchj@mail2.sysu.edu.cn}\\

\end{document}